\documentclass[10pt]{amsart}
\usepackage{amsmath}
\usepackage{amssymb}
\usepackage{amsthm}
\usepackage{bbm}
\usepackage[all]{xy}
\usepackage{dsfont}

\newtheorem{Thm}{Theorem}

\newtheorem{theorem}{Theorem}[section]
\newtheorem{lemma}[theorem]{Lemma}

\newtheorem{corollary}[theorem]{Corollary}

\newtheorem{proposition}[theorem]{Proposition}

\newtheorem{example}[theorem]{Example}

\theoremstyle{definition}
\newtheorem{remark}[theorem]{Remark}
\newtheorem{definition}[theorem]{Definition}
\newenvironment{proaf}{{\noindent \sc Proof} }{\mbox{ }\hfill$\Box$
                        \vspace{1.5ex}
                        \par}

\newenvironment{equationth}{\stepcounter{theorem}\begin{equation}}{\end{equation}}

\def\Z{\mathbb Z}
\def\C{\mathbb C}

\def\0{\underline 0}

\begin{document}

\title {\bf  On the Chern classes of singular complete intersections}

\author{Roberto Callejas-Bedregal}
\address{Centro de Ci\^encias Exatas e da Natureza,  Universidade Federal da Para\'{i}ba-UFPb, Jo\~ao Pessoa, PB - Brasil.}
\email{roberto@mat.ufpb.br}

\author{Michelle F. Z. Morgado}
\address{Instituto de Bioci\^encias Letras e  Ci\^encias Exatas.\\Universidade Estadual Paulista-UNESP, Brasil.}
\email{mmorgado@ibilce.unesp.br}

\author{Jos\'e Seade}
\address{Instituto de Matem\'aticas, Universidad Nacional Aut\'onoma de M\'exico.}
\email{jseade@im.unam.mx}

\thanks{Research partially supported by CAPES, CNPq and FAPESP,
Brazil, and by FORDECYT-CONACYT and PAPIIT-UNAM, Mexico.}

\keywords{Complete intersections, Milnor classes,
Whitney stratifications, Schwartz-MacPherson classes, Fulton-Johnson classes.}
 \subjclass{Primary ;
14C17, 55N45, 14M10 Secondary; 14B05, 32S20  }

\maketitle

\begin{abstract}
We   consider two classical  extensions  for singular varieties of the usual Chern classes of complex manifolds, namely  the total Schwartz-MacPherson   and  Fulton-Johnson classes, $c^{SM}(X)$ and   $c^{FJ}(X)$. Their difference (up to sign) is  the total Milnor class ${\mathcal M}(X)$,  a generalization of the Milnor number for varieties with arbitrary singular set.
We get first Verdier-Riemann-Roch type formulae for the total classes $c^{SM}(X)$ and   $c^{FJ}(X)$, and use these to
 prove a surprisingly simple formula for  the total Milnor class when $X$ is defined by a finite number of local complete intersection $X_1,\cdot \ldots \cdot,X_r$ in a complex manifold, satisfying  certain transversality conditions.
As applications  we obtain a Parusi\'{n}ski-Pragacz type formula and an
Aluffi type formula for the Milnor class, and a description of the Milnor classes of $X$  in terms of the global L\^e classes of
the $X_i$.
\end{abstract}

\section*{Introduction}
There are various different notions of Chern classes for singular varieties, each
having its own interest and characteristics. Perhaps the most
important of these are the total Schwartz-MacPherson class
$c^{SM}(X)$ and the total Fulton-Johnson class $c^{FJ}(X)$. In the complex analytic context these are elements in the
homology ring $H_{2*}(X, \Z)$ and in the algebraic context these
are elements in the Chow group $A_*(X)$.
Both of these classes
$c^{SM}(X)$ and $c^{FJ}(X)$ are defined by means of an embedding of $X$ in some complex manifold $M$, but they turn out to be  independent of the choice of embedding; when  $X$  is non-singular these are the Poincar\'e duals of the usual Chern classes. By definition the total Milnor class  of $X$ is:

\begin{equationth}\label{def Milnor class}
{\mathcal M}(X):=(-1)^{{\rm dim} X}\left(c^{FJ}(X)-c^{SM}(X) \right).
\end{equationth}

Milnor classes are a generalization of the classical Milnor number to varieties $X$ with arbitrary singular set. These have
support in the singular set ${\rm Sing}(X)$.
 There is a
Milnor class in each dimension from 0 to that of ${\rm Sing}(X)$.
In particular, when the singularities of $X$ are all  isolated, then there is only a $0$-degree Milnor
class which is an integer,  and if $X$  further is  a local complete intersection,  then this integer is the sum of the local Milnor
numbers of $X$ at its singular points (by \cite{SS, Suwa}).
 Milnor classes  are  important invariants   that encode much
information about the varieties in question, see for instance \cite{Aluffi,  Aluffi2, Alu-Mar,  BSS, BMS,  PP, Par-Pr, Sch5}.
Yet, most of the work on Milnor classes in the literature  is for  hypersurfaces, the
 complete intersection case being much harder ({\it cf.}  \cite {BLSS2, BMS-2, MSS}): that is the setting we envisage here. Our work is somehow inspired by the product formulas for the Milnor class of Ohmoto and Yokura in \cite{O-Y}.
 We prove:

 \begin{Thm}\label{theorem-1}
  Let $M$ be an $n$-dimensional compact complex
analytic manifold and let $\{E_{1},\cdot \ldots \cdot,E_r\}$, $
r\ge 1$, be
holomorphic vector bundles  over $M$ of ranks $d_{i}\ge1$. For each $i =1,\cdot \ldots \cdot r$,  let $X_i$ be the $(n-d_{i})$-dimensional local
complete intersection in $M$  defined by the zeroes of a regular section
 $s_{i}$ of $E_{i}$. Assume further that the   $X_i$ are equipped with
 Whitney stratifications ${\mathcal S}_i$ such that all the intersections amongst strata in the various $X_i$ are transversal. Set $X= X_1 \cap \cdot \ldots \cdot \cap X_r$, a local complete intesection of dimension $n -d_1 -\cdot \ldots \cdot -d_r$. Then:
 \begin{itemize}
\item [(i)] $\; \;  \;c^{SM}(X)=  c\left( \left(
TM|_{X}\right)^{\oplus r-1} \right)^{-1} \, \cap \,\;\Big(c^{SM}(X_{1}) \cdot \ldots \cdot c^{SM}(X_{r})\Big);$
\item [(ii)]  $\; \;c^{FJ}(X)=  c\left( \left(
TM|_{X}\right)^{\oplus r-1} \right)^{-1} \, \cap \,\;\Big(c^{FJ}(X_{1}) \cdot \ldots \cdot c^{FJ}(X_{r})\Big);$ and therefore
\item [(iii)]  ${\mathcal M}(X)=(-1)^{{\rm dim} X} \, c\left( \left(
TM|_{X}\right)^{\oplus r-1} \right)^{-1} \, \cap \,\;\Big(
c^{FJ}(X_{1}) \cdot \ldots \cdot c^{FJ}(X_{r}) - c^{SM}(X_{1}) \cdot \ldots \cdot c^{SM}(X_{r})\Big).$
\end{itemize}
 \end{Thm}

 The transversality condition in this Theorem can be relaxed (see section \ref{sec. main lemma}).
 Similar transversality conditions were used in  \cite{Sch5}  to prove a refined intersection formula for the Chern-Schwartz-MacPherson classes.

The proof of Theorem  \ref{theorem-1} takes most of this article.
The first step  is proving  Verdier-Riemann-Roch type  formulae for the  Schwartz-MacPherson,
 the Fulton-Johnson  and, therefore, for the Milnor classes of local complete intersections. In the last section of this article
we give various  applications. The first is  Theorem \ref{main-lemma} that
describes the Milnor class of $X$
  in terms of the Milnor  and the Schwartz-MacPherson classes of the $X_i$ and the Chern classes of $M$ restricted to $X$.
For instance, for $r=2$ we get the beautiful formula:
$${\mathcal M}(X)=c\left( \left( TM|_{X}\right) \right)^{-1}\cap
 \Big((-1)^{n}{\mathcal
M}(X_1)\cdot {\mathcal M}(X_2)+ (-1)^{d_{1}} c^{SM}(X_1)\cdot  {\mathcal
M}(X_2)+(-1)^{d_{2}}{\mathcal M}(X_1)\cdot  c^{SM}(X_2)\Big).$$
For $r= 3$ we get: {\small
$${\mathcal M}(X)= c\left( \left( TM|_{X}\right)^{\oplus 2}
\right)^{-1}\cap \;\; \Big({\mathcal M}(X_1)\cdot {\mathcal M}(X_2)\cdot
{\mathcal M}(X_3)+ (-1)^{(d_{1}+d_{2})} c^{SM}(X_1)\cdot
c^{SM}(X_2)\cdot  {\mathcal M}(X_3) +$$
$$+(-1)^{(d_{1}+d_{3})} c^{SM}(X_1)\cdot {\mathcal
M}(X_2)\cdot  c^{SM}(X_3)+(-1)^{(d_{2}+d_{3})} {\mathcal M}(X_1)\cdot
c^{SM}(X_2)\cdot  c^{SM}(X_3)+$$
$$+(-1)^{(n-d_{1})} c^{SM}(X_1)\cdot {\mathcal M}(X_2)\cdot {\mathcal
M}(X_3)+(-1)^{(n-d_{2})} {\mathcal M}(X_1)\cdot c^{SM}(X_2)\cdot {\mathcal
M}(X_3)+$$
$$+(-1)^{(n-d_{3})} {\mathcal M}(X_1)\cdot {\mathcal M}(X_2)\cdot  c^{SM}(X_3)
\Big),$$} \hskip-3pt and so on.
This highlights why understanding the Milnor classes of complete intersections is {\it a priori} far more difficult than in the hypersurface case, though the   formula in Theorem  \ref{theorem-1} is surprisingly simple.

We then restrict the
 discussion to the case where the bundles in question are all line bundles $L_i$. We get two interesting  applications of Theorem \ref{main-lemma}:

\vspace{0,2cm}  \noindent
{\bf i)}
A Parusi\'nski-Pragacz type formula for local
complete intersections as above (Corollary \ref{P-P}).
This expresses  the
Milnor classes using only  Schwartz-MacPherson classes, and it answers positively the expected description given
by Ohmoto and Yokura in \cite{O-Y} for the  total Milnor class of a local
complete intersection.   We  notice that a different generalization of the Parusi\'nski-Pragacz  formula for complete intersections has been given recently in \cite{MSS}.

\vspace{0,2cm}  \noindent
{\bf ii)}  A   description of the total Milnor class of the local complete intersection $X$  in the vein of  Aluffi's formula in \cite{Aluffi} for hypersurfaces, using  Aluffi's $\mu$-classes (Corollary \ref{Aluffi}).

\vskip.2cm

This work is a refinement of our unpublished article \cite{BMS-1} (cf. also \cite {BMS-2}). We are indebted to the referee and to J\"{o}rg Sch\"{u}rmann for valuable suggestions.
We are also grateful  to Nivaldo Medeiros and Marcelo Saia for fruitful
 conversations.

\section{Chern classes and the diagonal embedding}
\subsection{Derived categories}
 We  assume some basic knowledge  on derived categories as described  for instance in
\cite{Dimca}.
If $X$ is a complex analytic space then ${\mathcal D}^{b}_{c}(X)$
denotes the derived category of bounded constructible complexes of
sheaves of $\C$-vector spaces on $X$. We denote the objects of
${\mathcal D}^{b}_{c}(X)$ by something of the form $F^{\bullet}$. The
shifted complex $F^{\bullet}[l]$ is defined by
$(F^{\bullet}[l])^{k}=F^{l+k}$ and  its differential is
$d^{k}_{[l]}=(-1)^{l}d^{k+l}$. The constant sheaf $\C_{X}$ on $X$
induces an object $\C_{X}^{\bullet} \in {\mathcal D}^{b}_{c}(X)$ by
letting $\C_{X}^{0}=\C_{X}$ and $\C_{X}^{k}=0$ for $k\neq 0$.
If $h:X\rightarrow \C$ is an analytic map and $F^{\bullet}\in
{\mathcal D}^{b}_{c}(X)$, then we denote the sheaf of vanishing cycles
of $F^{\bullet}$  with respect to $h$ by $\phi_{h}F^{\bullet}$.
For $F^{\bullet}\in {\mathcal D}^{b}_{c}(X)$ and $p \in X$, we denote
by ${\mathcal H}^{*}(F^{\bullet})_{p}$ the stalk cohomology of
$F^{\bullet}$ at $p$, and by $\chi(F^{\bullet})_{p}$ its Euler
characteristic. That is,
$$\chi(F^{\bullet})_{p}=\sum_{k}(-1)^{k}{\rm dim}_{\C}{\mathcal
H}^{k}(F^{\bullet})_{p}.$$
We also denote by $\chi(X,F^{\bullet})$ the Euler characteristic of
$X$ with coefficients in $F^{\bullet}$, {\it i.e.},
$$\chi(X,F^{\bullet})=\sum_{k}(-1)^{k}{\rm dim}_{\C}\;\mathbb{H}^{k}(X,F^{\bullet}),$$
where $\mathbb{H}^{*}(X,F^{\bullet})$ denotes the hypercohomology
groups of $X$ with coefficients in $F^{\bullet}$.
When $F^{\bullet}\in {\mathcal D}^{b}_{c}(X)$ is ${\mathcal
S}$-constructible, where ${\mathcal S}$ is a Whitney stratification of
$X$, we  denote it by $F^{\bullet}\in {\mathcal D}^{b}_{{\mathcal S}}(X)$. Setting
$\chi(F^{\bullet}_{S})=\chi(F^{\bullet})_{p}$ for an arbitrary point
$p \in S$,
we have
\cite[Theorem 4.1.22]{Dimca}:
\begin{equation}\label{EulerCharact}\chi(X,F^{\bullet})=\sum_{S\in {\mathcal
S}}\chi(F^{\bullet}_{S})\chi(S)\,.\end{equation}

For a subvariety $X$ in a complex manifold  $M$ we
denote its conormal variety by
$T^{*}_{X}M$. That is, $$T^{*}_{X}M:={\rm closure}\;\{ (x, \theta) \in T^{*}M\;|\; x
\in X_{{\rm reg}}\;{\rm and}\; \theta_{|_{T_{x} X_{{\rm
reg}}}}\equiv 0\}\,,$$ where $T^{*}M$ is the cotangent bundle  and
$X_{{\rm reg}}$ is the regular part.
One has (see
\cite{GM}):

\begin{definition}
Let $ X $ be an analytic subvariety of a complex manifold $ M $, $\{
S_{\alpha} \}$ a Whitney stratification of $M$ adapted to $ X $ and
$x\in S_\alpha$ a point in $X$. Consider $g:(M,x)\rightarrow (\C,0)$
a germ of holomorphic function such that $d_{x}g $ is a {\it
non-degenerate covector} at $x$ with respect to the fixed
stratification. That is, $d_{x}g \in T^{*}_{S_\alpha}M$ and $d_{x}g
\not\in T^{*}_{S^{'}}M$  for all stratum $S^{'} \neq S_\alpha$.
Let $N$ be  a germ of a closed complex submanifold of $M$ which is
transversal to $S_\alpha$, with $N \cap S_\alpha=\{ x\}$. Define the
{\it complex link }  $l_{S_\alpha}$  of $S_\alpha$ by:
 \begin{center}$l_{S_\alpha}:= X\cap N \cap
B_{\delta}(x)\cap \{g=w\}\quad{\rm for}\;0<|w|<\!\!< \delta<\!\!<
1.$\end{center}
The {\it normal Morse datum} and  the {\it normal Morse index} of the stratum  $S_\alpha$
are, respectively:
$$NMD(S_\alpha):=(X\cap N \cap B_{\delta}(x),l_{S_\alpha}) \quad \hbox{and} \quad
\eta(S_\alpha,F^{\bullet}):=\chi(NMD(S),F^{\bullet})\,,$$
where the right-hand-side means the Euler characteristic of the
relative hypercohomology. \end{definition}

In fact, the slice $N$ normal to the stratum $S_\alpha$ at $x$ is transversal to all other stratum that contain $x$ in their closure, by Whitney regularity. Therefore the Whitney stratification on $X$ induces a Whitney stratification on $NMD(S_\alpha)$. Hence the sheaf $F^{\bullet}$ restricted to $NMD(S_\alpha)$ is constructible and therefore the relative hypercohomology is well-defined.

By  \cite[Theorem 2.3]{GM}  we
get that  $\eta(S_\alpha,F^{\bullet})$ does not depend on
the choices of $x\in S_\alpha,\; g$ and $N$. By \cite[p. 283]{Sch4}, the  normal Morse index $\eta(S_\alpha,F^{\bullet})$  can be computed in terms of sheaves of vanishing cycles as
\begin{equationth}\label{Dimca}
\eta(S_\alpha,F^{\bullet})=-\chi(\phi_{g|_{N}}(F^{\bullet}|_{N})).
\end{equationth}
By \cite[Remark 2.4.5(ii)]{Dimca} this can also be expressed as:
\begin{equation}\label{RelativeEuler}\eta(S_\alpha,F^{\bullet})=\chi(X\cap N \cap
B_{\delta}(x),F^{\bullet})-\chi(l_{S_\alpha},F^{\bullet})\,.\end{equation}

\subsection{Chern classes for singular varieties}
 From now on, let $M$ be an $n$-dimensional compact complex
analytic manifold and let $E$ be a holomorphic vector bundle over
$M$ of rank $d$. Let $X$ be the zero scheme of a regular holomorphic
section of $E$, which is an $(n-d)$-dimensional local complete
intersection. Consider {\it {the virtual bundle}} $\tau(X;M):= T
M|_{_{X}}-E|_{_{X}}$, where $T M$ denotes the tangent bundle of $
M$ and the difference is in the  K-theory of $X$. The element $\tau(X;M)$ actually is independent of $M$ (see \cite[Appendix B.7.6.]{Ful}) and is called {\it {the virtual tangent bundle of}} $X$. {\it The Fulton-Johnson
homology class of} $X$ is defined by the Chern class of  $\tau(X;M)$
via the Poincar\'e morphism, that is (cf. \cite{Ful}):
$$c^{FJ}(X;M)=c(\tau(X;M))\cap
[X]:=c^{}(T M|_X) c^{}(E|_X)^{-1}\cap [X].$$
 For simplicity we denote the virtual bundle and the Fulton-Johnson classes simply by $\tau(X)$ and $c^{FJ}(X)$.

Consider now the Nash blow up $\tilde{X} \stackrel{\nu}{\rightarrow}
 X$ of  $X$, its Nash bundle ${\tilde T} \stackrel{\pi}{\rightarrow}
\tilde{ X} $ and the Chern classes of $\tilde{T}$, $c^{j}(\tilde{T})
\in H^{2j}(\tilde{ X})$, $j=1,\ldots,n$. {\it The Mather classes} of $X$ are:
$$c^{Ma}_{k}( X):=v_{*}(c^{n-d-k}(\tilde{T})\cap [\tilde{ X}])\in H_{2k}( X),\;\;k=0,\ldots,n \,.$$
We equip $X$ with a Whitney stratification $X_{\alpha}$.
The MacPherson classes are obtained from the Mather classes by considering
appropriate ``weights" for each stratum, determined by the local
Euler obstruction ${\rm Eu}_{{ X_{\alpha}}}(x)$.  This is an integer associated   in \cite{MacP}
to each point $x \in X_{\alpha}$.  It is proved in \cite{MacP} that
there exists a unique set of integers $b_{\alpha}$,
 for which the
equation $\sum b_{\alpha} {\rm Eu}_{\bar{ X}_{\alpha}}(x)=1$ is
satisfied for all points $x \in  X$. Here, $\bar{ X}_{\alpha}$ denotes the closure of the stratum, which is itself analytic; the sum runs over all
strata $ X_{\alpha}$ containing $x$ in their closure.
Then {\it the  MacPherson class of degree} $k
$ is defined by
$$c^{M}_{k}( X):=\sum
b_{\alpha}\;i_{*}(c^{Ma}_{k}(\bar{ X}_{\alpha})),$$
where
$i:\bar{ X}_{\alpha}\hookrightarrow  X$ is the inclusion map.
We remark that by \cite{BS}, the MacPherson classes coincide,
up to Alexander duality, with the classes defined by
M.-H. Schwartz in \cite{Sch}. Thus, following the modern
literature (see for instance \cite{Par-Pr, BLSS2, BSS}), these  are called
Schwartz-MacPherson classes of $X$ and  denoted
 by $c^{SM}_{k}(X)$.

\begin{definition}
The total Milnor class of $X$ is (see \cite{BSS,Par-Pr}):
$${\mathcal M}(X):=(-1)^{n-d}\left(c^{FJ}(X)-c^{SM}(X) \right).$$
\end{definition}

\subsection{Milnor classes and the diagonal embedding}
 Given a manifold  $M$  as before,  set $M^{(r)}:=M\times \ldots \times M$, $r$ times. We let $E$ be a holomorphic vector
bundle over $M^{(r)}$ of rank $d$. Consider $\Delta: M\rightarrow
M^{(r)}$ the diagonal morphism, which is a regular embedding of
codimension $nr-n$. Let $t$ be a regular holomorphic section of $E$.
The set of the zeros of $t$ is a closed subvariety $Z(t)$ of
$M^{(r)}$ of dimension $nr-d$. Consider $Z(\Delta^{*}(t))$ the set
of the zeros of the pull back section of $t$ by
$\Delta$.

 Following \cite[Chapter 6]{Ful} we have that
$\Delta$ induces the refined Gysin homomorphism
$$\Delta^{!}:H_{2k}(Z(t))\rightarrow H_{2(k-nr+n)}(Z(\Delta^{*}(t))).$$
The refined intersection product is defined by
${\alpha}_1  \cdot \ldots \cdot {\alpha}_r:=\Delta^{!}({\alpha}_1\times\ldots \times
{\alpha}_r)$.
For the usual homology this is defined by duality between
homology and cohomology:
$$\Delta^{!}\! = \Delta^{*}\! : H_{2k}(Z(t);\Z) \simeq H^{2(nr-k)} ({Z(t)};\Z)\rightarrow H^{2(nr-k)}({Z(\Delta^{*}(t))};\Z)  \simeq H_{2(k-nr+n)}(Z(\Delta^{*}(t));\Z).$$

\begin{remark}\label{Fu}
\begin{enumerate}

\item In \cite[Proposition 14.1, (c) and (d)(ii)]{Ful} it is proved that if $f:X'\rightarrow X$ is a local complete intersection morphism between purely dimensional schemes, $E$ is a vector bundle on $X,$  $s$ is a regular section of $E$ and $s'=f^*s$ is the induced section on $f^*E,$ then $f^![Z(s)]=[Z(s')],$ where $f^!$ is the refined Gysin homomorphism induced by $f.$

\vspace{0.3cm}\item In \cite[Proposition 6.3]{Ful} it is proved that if $\iota:X\rightarrow Y$ is a regular embedding and $F$ is a vector bundle on $Y,$ then $\iota^!(c_m(F)\cap \alpha)=c_m(\iota^*F)\cap \iota^! \alpha,$ for all $\alpha\in H_{2k}(Y,\Z).$  Applying this result to the diagonal morphism $\Delta: M\rightarrow M^{(r)},$ which is a regular embedding, we have that for any vector bundle $F$ on $M^{(r)}$ holds that $\Delta^!(c_m(F) \cap \alpha)=c_m(\Delta^*F) \cap\Delta^!(\alpha)$ for all $\alpha \in H_{2k}(M^{(r)};\Z)$
and $m\geq 0.$

\end{enumerate}
\end{remark}

These two remarks are used for following  a Verdier-Riemann-Roch type theorem for the Fulton-Johnson classes:

\begin{proposition}\label{t:1} The refined Gysin morphism satisfies: $$\Delta^{!}\left( \;c^{FJ}(Z(t))\;\right)=
c\left( \left( TM|_{Z(\Delta^{*}t)}\right)^{\oplus r-1} \right)\cap c^{FJ}(Z(\Delta^{*}t))\,.$$\end{proposition}

\begin{proof}
By definition of the Fulton-Johnson class we have $$\Delta^{!}\;c^{FJ}(Z(t))=\Delta^{!}\left(
c\left( TM^{(r)}|_{Z(t)}\right) c\left(
{E}|_{Z(t)}\right)^{-1} \cap [Z(t)] \right).$$

Applying Remark \ref{Fu} (2) to the diagonal morphism $\Delta: M\rightarrow M^{(r)},$ which is a regular embedding,  and using the virtual bundle we have that
$$\Delta^{!}\;c^{FJ}(Z(t))=c\left(\Delta^{*}\left(
TM^{(r)}|_{Z(t)}\right)\right) c\left(
\Delta^{*}\left({E}|_{Z(t)}\right)\right)^{-1} \cap
\Delta^{!}\;[Z(t)].$$ Note that
$\Delta^{*}\left({E}|_{Z(t)}\right)=\Delta^{*}{E}|_{Z(\Delta^{*}t)}$
and, applying Remark \ref{Fu} (1) to the diagonal morphism $\Delta: M\rightarrow M^{(r)}$, which is a local complete intersection morphism, and to the regular section $t$ of $E$ we obtain that $\Delta^![Z(t)]=[Z(\Delta^*(t))].$

 Moreover, since $\Delta^{*} TM^{(r)}= TM\oplus
\ldots \oplus TM$, we have
$$c\left(\Delta^{*}\left(
TM^{(r)}|_{Z(t)}\right)\right)=c\left(\left(
TM|_{Z(\Delta^{*}t)}\right)^{\oplus\;r}\right) \,$$ and the result
follows.
\end{proof}

Let $F(M)$ be the free abelian group of constructible functions on
$M$ with respect to a Whitney stratification $\{ S_{\alpha} \}$. It is proved in \cite{MacP}  that every element $\xi$ in $F(M)$ can be written uniquely in the form:
$$\xi = \sum n_W {\rm Eu}_W \,,$$
for some appropriate subvarieties $W$ and integers $n_W$.
Let  $L( M)$  be the free abelian group of all cycles generated by
the conormal spaces $T^{*}_{W} M$, where $W$ varies over all
subvarieties of $ M$. Given $\xi \in  F(M)$
 define
an element $Ch(\xi)$ in $L(M)$ by:
\begin{equationth}\label{ch}
Ch(\xi):=\sum_{\alpha} (-1)^{\dim W} n_W\cdot  T^{*}_{W} M \,.
\end{equationth}
This
induces an isomorphism $Ch : F(M) \rightarrow L(M)$.
Define the map  $cn:Z( M)\rightarrow L( M)$ by $cn(X):=T^{*}_{X}
M$. Clearly, this is also an isomorphism.
We know from \cite[Section 3]{BMS} that we have a commutative
diagram:

\begin{equation}\label{diagrama}
\xymatrix {  &Z(M) \ar[r]^{\check{E}u} \ar[d]^{cn} & F(M)\ar[d]^{\mbox{id}}\\
&  L(M)  \ar[r]^{Ch}  & F(M) }\end{equation}
The commutativity  of this diagram amounts to
saying:
$$\beta=\sum_{\alpha}\eta(S_{\alpha},\beta) \cdot
{E}u_{S_{\alpha}},$$ for any function $\beta:X\rightarrow
\Z$ which is constructible for the given Whitney stratification, where
$\eta(S_\alpha,\xi)=\eta(S_\alpha,F^{\bullet})$, with $F^{\bullet}$ being the
complex of sheaves such that $\chi( F^{\bullet})_{p}=\xi(p)$.
Substituting in equation (\ref{ch}) we get:
\begin{equationth}\label{nueva}
Ch(\xi):=\sum_{\alpha} (-1)^{\dim S_\alpha}\eta(S_\alpha,\xi)\cdot  T^{*}_{\overline{S}_\alpha} M \,.
 \end{equationth}

\vspace{0,2cm}
Now consider the projectivized cotangent bundles $\mathbb{P}(T^{*}M)$ and  $\mathbb{P}(T^{*}(M^{(r)}))$; we denote by $\mathbb{P}((T^{*}M)^{\oplus r})\,$
the bundle $\mathbb{P}(T^{*}M\oplus \ldots \oplus T^{*}M)$.
Notice that one has a fibre square diagram (see \cite[pag. 428]{Ful}):

\begin{equation}\label{fiber square}
\xymatrix {    \mathbb{P}((T^{*}M)^{\oplus r})
\ar[r]^{\delta} \ar[d]_{p} & \mathbb{P}(T^{*}(M^{(r)}))  \ar[d]^{\pi^{(r)}} \\ M \ar[r]^{\Delta}  & M^{(r)}  }
\end{equation}
where $\pi^{(r)}$ is the natural proper map. Let $i: \mathbb{P}(T^{*}M) \to \mathbb{P}((T^{*}M)^{\oplus r}) $ be the morphism
 induced by the diagonal embedding
 $T^{*}M \to T^{*}M \oplus \ldots \oplus T^{*}M$.

\begin{proposition} \label{transversal} Let $\beta$ be a constructible function on $M^{(r)}$ with respect to a Whitney stratification $\{{\mathcal T}_\gamma\}$, which we assume transversal to
$\Delta(M)$. Then:
$$\delta^{!}\;[\mathbb{P}(Ch(\beta))]=(-1)^{nr-n}\; i_{*}\;[\mathbb{P}(Ch(\Delta^{*}(\beta)))].$$
\end{proposition}

\begin{proof} Since the stratification $\{{\mathcal T}_\gamma\}$  is transversal to $\Delta(M)$, we have that
$\{\Delta^{-1}({\mathcal T}_\gamma)\}$ is a Whitney stratification of
$M$ with respect to which $\Delta^{*}(\beta)$ is a constructible function. Moreover, if $T$ is a normal slice of $\Delta^{-1}({\mathcal T}_\gamma)$ at $x$ then $\Delta(T)$ is a normal slice of
${\mathcal T}_\gamma$ at $(x,\ldots,x).$ Set $N=\Delta(T).$

By equations (\ref{Dimca}) and (\ref{nueva}) we have  $\,\mathbb{P}(Ch(\beta))=\sum
m_{\gamma}\mathbb{P}\left(\overline{T_{{\mathcal T}_{\gamma}}^{*}M^{(r)}}\right)\,,$
 where
$m_{\gamma}:= (-1)^{nr-d-1}\chi\left(\phi_{f| N}F^{\bullet}|_{N}
\right)_{(x,\ldots,x)}$, and $F^{\bullet}$ is the bounded complex of sheaves such that $\chi( F^{\bullet})_{p}=\beta(p)$ and $f:
(M^{(r)},(x,\ldots,x))\rightarrow (\C,0)$ is a germ such that $d_{(x,\ldots,x)}f$ is a non-degenerate covector at $(x,\ldots,x)$ with respect to $\{{\mathcal T}_{\gamma}\}$.

Analogously,
$$\mathbb{P}(Ch(\Delta^{*}(\beta))=\sum_{\gamma}
n_{\gamma}\mathbb{P}\left(\overline{T_{\Delta^{-1}({\mathcal T}_{\gamma})}^{*}M}\right),$$
where $n_{\gamma}:=
(-1)^{n-d-1}\chi\left(\phi_{g|T}G^{\bullet}|_{T}
\right)_{x}$, where
$G^{\bullet}=\Delta^{*}F^{\bullet},$ which is the
bounded complex of sheaves such that $\chi(
G^{\bullet})_{q}=\Delta^{*}(\beta)(q),$ and $g:
(M,x) \rightarrow (\C,0)$  is a germ such that $d_xg$ is a non-degenerate covector at $x$ with respect to $\{\Delta^{-1}({\mathcal T}_{\gamma})\}$. Notice that we can take $g=\Delta^{*}f$ since these definitions do not depend on the choices of $g.$

Notice that $\Delta|_T:T\rightarrow N$ is an isomorphism. Hence $$\phi_{\Delta^{*}(f| N)}\Delta^{*}(F^{\bullet}|_{ N})\simeq \Delta^{*}\left(\phi_{f|N}\left(F^{\bullet}|_{ N}\right)\right)\,.$$ \noindent But clearly
$\phi_{\Delta^{*}(f| N)}\Delta^{*}(F^{\bullet}|_{ N})=\phi_{g| T}G^{\bullet}|_{T},$ thus $$\chi\left(\phi_{g| T}G^{\bullet}|_{T}
\right)_{x}=\chi \left(\Delta^{*}\left(\phi_{f|N}\left(F^{\bullet}|_{ N}\right)\right)\right)_{x}=\chi\left(\phi_{f|N}F^{\bullet}|_{N}
\right)_{(x,\ldots,x)}\,.$$ Therefore
\begin{equation}\label{6}m_{\gamma}=(-1)^{nr-n} n_{\gamma}. \end{equation}

 Proposition \ref{transversal} is now an immediate consequence of the next  lemma:
 \end{proof}
\begin{lemma} \label{7} One has:
$$\delta^{!}\;\left[\mathbb{P}\left(\overline{T_{{\mathcal T}_{\gamma}}^{*}M^{(r)}}\right)\right]=
i_{*}\;\left[\mathbb{P}\left(\overline{T_{\Delta^{-1}({\mathcal T}_{\gamma})}^{*}M}\right)\right].$$
\end{lemma}

\begin{proof}
Consider the projectivized cotangent bundles $\mathbb{P}(T^{*}M)$ and  $\mathbb{P}(T^{*}(M^{(r)}))$; we denote by $\mathbb{P}((T^{*}M)^{\oplus r})\,$
the bundle $\mathbb{P}(T^{*}M\oplus \ldots \oplus T^{*}M)$.
Notice that one has a fibre square
 diagram :
$$
\xymatrix {    \mathbb{P}\left(\overline{\Delta^{*}(T^{*}_{{\mathcal T}_\gamma}M^{(r)})}\right)\ar[r]^{{\delta}'}\ar[d]_{j'} & \mathbb{P}\left(\overline{T^{*}_{{\mathcal T}_\gamma}M^{(r)}}\right) \ar[d]^{j} \\
\mathbb{P}(\Delta^{*}T^{*}(M^{(r)}))
\ar[r]^{\delta} \ar[d]_{p} & \mathbb{P}(T^{*}(M^{(r)}))  \ar[d]^{\pi^{(r)}} \\ M \ar[r]^{\Delta}  & M^{(r)}  }
$$

where $\pi^{(r)}$ is the natural proper map.

Notice that $$\mathbb{P}(\Delta^{*}T^{*}(M^{(r)}))=\mathbb{P}((T^{*}M)^{\oplus r})$$ and
$$\mathbb{P}\left(\overline{\Delta^{*}(T^{*}_{{\mathcal T}_\gamma}M^{(r)})}\right)= \mathbb{P}\left(\overline{T^*_{\Delta^{-1}({\mathcal T}_\gamma)}M} \right).$$

Thus $j'$ is induced by the diagonal embedding $i.$

Notice that $\Delta,\, \delta$ and $\delta'$ are regular embeddings of codimension $nr-n.$ Hence
$$N_{\mathbb{P}\left(\overline{\Delta^{*}(T^{*}_{{\mathcal T}_\gamma}M^{(r)})}\right)}\mathbb{P}\left(\overline{T^{*}_{{\mathcal T}_\gamma}M^{(r)}}\right)={j'}^*N_{\mathbb{P}(\Delta^{*}T^{*}(M^{(r)}))}\mathbb{P}(T^{*}(M^{(r)})).$$

Therefore $$\delta^{!}\;\left[\mathbb{P}\left(\overline{T_{{\mathcal T}_{\gamma}}^{*}M^{(r)}}\right)\right]={j'}_*\left[\mathbb{P}\left(\overline{\Delta^{*}(T^{*}_{{\mathcal T}_\gamma}M^{(r)})}\right)\right]=
i_{*}\;\left[\mathbb{P}\left(\overline{T_{\Delta^{-1}({\mathcal T}_{\gamma})}^{*}M}\right)\right].$$

Hence the result follows.

\end{proof}

\begin{corollary}\label{L:4} Let $Z(t)$ be as in Proposition \ref{t:1}. Assume that $Z(t)$ admits a Whitney stratification $\{{\mathcal T}_\gamma\}$ transversal to $\Delta(M)$. Then:
$$\delta^{!}\;[\mathbb{P}(Ch({\mathbbm 1}_{Z(t)}))]=(-1)^{nr-n}\; i_{*}\;[\mathbb{P}(Ch({\mathbbm 1}_{Z(\Delta^{*}t)}))],$$
where ${\mathbbm 1}_{(\;)}$ denotes the characteristic function.
\end{corollary}

\begin{remark}\label{Fulton}
\begin{enumerate}
\item  In \cite[Theorem 6.2. (a)]{Ful} it is proved the following: Consider a fiber square diagram

$$\xymatrix {    X'
\ar[r]^{\iota'} \ar[d]_{q} & Y'  \ar[d]^{p} \\ X \ar[r]^{\iota}  & Y  },$$
\noindent where $\iota$ is a regular embedding of codimension $d$ and $p$ is a proper morphism, then $\iota^!p_*(\alpha)=q_*(\iota^!(\alpha)),$ for all $\alpha \in H_{2k}(Y',\Z).$ Also in \cite[Theorem 6.2. (c)]{Ful} it is proved that if $\iota'$ is also a regular embedding of codimension $d,$ then $\iota^!(\alpha)=\iota'^!(\alpha),$ for all $\alpha \in H_{2k}(Y',\Z).$

\vspace{0.3cm}\item In  \cite[Equation (14)]{Par-Pr} Parusi\'{n}ski and Pragacz gave the following description of the Schwartz-MacPherson classes: Let $Z$ be a smooth complex manifold, let  $V$ be a closed subvariety of $Z$ and $\pi:\mathbb{P}(T^*Z)\rightarrow Z$ be the projectivized cotangent bundle of $Z,$ then the Schwartz-MacPherson class  of $V$ is given by
$$c^{SM}(V)=(-1)^{\dim Z-1}c(TZ|_V)\cap \pi_*\left(c({\mathcal{O}(1)})^{-1}\cap [\mathbb{P}(Ch({\mathbbm 1}_{V}))]\right),$$
\noindent where ${\mathcal{O}(1)}$ is the tautological line bundle of $\mathbb{P}(T^*Z).$

\end{enumerate}
\end{remark}

\begin{theorem}\label{T:2} With the assumptions of Corollary \ref{L:4} we have: $$\Delta^{!}\left( \;c^{SM}(Z(t))\;\right)=
c\left( \left( TM|_{Z(\Delta^{*}t)}\right)^{\oplus r-1} \right)\cap c^{SM}(Z(\Delta^{*}t))\,.$$\end{theorem}

\begin{proof}
Applying Remark \ref{Fulton} (2) to the projectivized cotangent
bundle $\pi^{(r)}:\mathbb{P}(T^{*}M^{(r)})\rightarrow M^{(r)}$ we
obtain
$$c^{SM}(Z(t))=(-1)^{nr-1}c\left( TM^{(r)}|_{Z(t)} \right) \cap
\pi_{*}^{(r)} \left( c({\mathcal O}_{r}(1))^{-1}\cap
[\mathbb{P}(Ch({\mathbbm 1}_{Z(t)}))]\right),$$
 where ${\mathcal O}_{r}(1)$ denotes the tautological line
bundle of $\mathbb{P}(T^{*}M^{(r)}).$

Applying Remark \ref{Fu} (2) we have that
\begin{equation}\label{8}{\small \Delta^{!}\;c^{SM}(Z(t))=(-1)^{nr-1}c\left(\Delta^{*}\left(
TM^{(r)}|_{Z(t)} \right)\right) \cap \Delta^{!}\pi_{*}^{(r)} \left(
c({\mathcal O}_{r}(1)))^{-1}\cap
[\mathbb{P}(Ch({\mathbbm 1}_{Z(t)}))]\right).}\end{equation}

Applying Remark \ref{Fulton} (1) to the fiber square diagram (\ref{fiber square}) we get that
\begin{equation}\label{9}\Delta^{!}\pi_{*}^{(r)} \left( c({\mathcal
O}_{r}(1)))^{-1}\cap [\mathbb{P}(Ch({\mathbbm 1}_{Z(t)}))]\right)= p_{*} \left(
\delta^{!}(c({\mathcal O}_{r}(1))^{-1}\cap
\;[\mathbb{P}(Ch({\mathbbm 1}_{Z(t)}))])\right).\end{equation}

Applying again Remark \ref{Fu} (2) \ we have,
\begin{equation}\label{100}\Delta^{!}\pi_{*}^{(r)} \left( c({\mathcal
O}_{r}(1)))^{-1}\cap [\mathbb{P}(Ch({\mathbbm 1}_{Z(t)}))]\right)= p_{*} \left(
c(\delta^{*}{\mathcal O}_{r}(1)))^{-1}\cap
\delta^{!}\;[\mathbb{P}(Ch({\mathbbm 1}_{Z(t)}))]\right).\end{equation}
Since $\delta^{*}{\mathcal O}_{r}(1)= {\mathcal
O}_{\mathbb{P}\left((T^{*}M)^{\oplus r} \right)}(1)$ is the
tautological line bundle on the projectivization
$\mathbb{P}((T^{*}M)^{\oplus r})\rightarrow M$, by Corollary \ref{L:4} and
the equations (\ref{8}), (\ref{9}) and (\ref{100}), we get:
$$\begin{array}{l}
\Delta^{!}\left( \;c^{SM}(Z(t))\;\right)= (-1)^{n-1}c\left( \left(
TM|_{Z(\Delta^{*}t)}\right)^{\oplus r} \right)\cap\\\quad \quad\quad\quad\quad\quad\quad \,\;\cap\; p_{*} \left(
c({\mathcal O}_{\mathbb{P}\left((T^{*}M)^{\oplus r}
\right)}(1))^{-1}\cap
i_{*}\;[\mathbb{P}(Ch({\mathbbm 1}_{Z(\Delta^{*}t)}))]\right).\end{array}$$

Hence, by the projection formula for proper morphism (see \cite[Theorem 3.2 (c)]{Ful}) we have that $$\begin{array}{l}
\Delta^{!}\left( \;c^{SM}(Z(t))\;\right)= (-1)^{n-1}c\left( \left(
TM|_{Z(\Delta^{*}t)}\right)^{\oplus r-1} \right)c\left( \left(
TM|_{Z(\Delta^{*}t)}\right) \right)\cap\\\quad \quad\quad\quad\quad\quad\quad \,\;\cap\; (p\circ i)_{*} \left(
c(i^*{\mathcal O}_{\mathbb{P}\left((T^{*}M)^{\oplus r}
\right)}(1))^{-1}\cap
\;[\mathbb{P}(Ch({\mathbbm 1}_{Z(\Delta^{*}t)}))]\right).\end{array}$$

Now, using the fact that $i^*{\mathcal O}_{\mathbb{P}\left((T^{*}M)^{\oplus r}
\right)}(1)={\mathcal O}_{\mathbb{P}\left((T^{*}M)
\right)}(1)$ and that $p\circ i =q:\mathbb{P}(T^{*}M)\rightarrow M$ is the projectivized cotangent morphism we have that

$$\begin{array}{l}
\Delta^{!}\left( \;c^{SM}(Z(t))\;\right)= (-1)^{n-1}c\left( \left(
TM|_{Z(\Delta^{*}t)}\right)^{\oplus r-1} \right)c\left( \left(
TM|_{Z(\Delta^{*}t)}\right) \right)\cap\\\quad \quad\quad\quad\quad\quad\quad \,\;\cap\; q_{*} \left(
c({\mathcal O}_{\mathbb{P}\left((T^{*}M)
\right)}(1))^{-1}\cap
\;[\mathbb{P}(Ch({\mathbbm 1}_{Z(\Delta^{*}t)}))]\right).\end{array}$$

Applying Remark \ref{Fulton} (2) to the
projectivize cotangent bundle $q:\mathbb{P}(T^{*}M)\rightarrow M$ we obtain that
$$c^{SM}(Z(\Delta^{*}t))=(-1)^{n-1}c\left( \left(
TM|_{Z(\Delta^{*}t)}\right) \right)\cap q_{*} \left(
c({\mathcal O}_{\mathbb{P}\left((T^{*}M)
\right)}(1))^{-1}\cap
\;[\mathbb{P}(Ch({\mathbbm 1}_{Z(\Delta^{*}t)}))]\right)$$

Hence we have that
$$\Delta^{!}\left( \;c^{SM}(Z(t))\;\right)=c\left( \left( TM|_{Z(\Delta^{*}t)}\right)^{\oplus r-1}
\right)\cap c^{SM}(Z(\Delta^{*}t)).$$

\end{proof}

Theorem \ref{T:2} is a Verdier-Riemann-Roch type  formula for the Schwarz-MacPherson classes (cf. \cite{Sch2}).
Analogously, the next result is a Verdier-Riemann-Roch type theorem for the Milnor classes. The proof is a straightforward application of Proposition  \ref{t:1} and Theorem \ref{T:2}.

\begin{corollary} \label{10}With the assumptions of Corollary \ref{L:4} we have: $$\Delta^{!}{\mathcal M}(Z(t))=(-1)^{nr-n}c\left( \left( TM|_{Z(\Delta^{*}t)}\right)^{\oplus
r-1} \right)\cap{\mathcal M}(Z(\Delta^{*}t))\,.$$\end{corollary}

\section{Intersection product formulas}\label{sec. main lemma}

As before, let $M$ be an $n$-dimensional compact complex
analytic manifold. Let $\{E_{i}\}$ be a finite collection of
holomorphic vector bundles  over $M$ of rank $d_{i}$, $1 \leq i\leq
r$. For each of these bundles, let  $s_{i}$ be a regular
holomorphic section
 and $X_{i}$ the $(n-d_{i})$-dimensional local
complete intersections  defined by the zeroes of
 $s_{i}$.

 In this section we assume that  we can equip the product  $X_1 \times \ldots \times X_r$ with a Whitney stratification such that
 the diagonal embedding $\Delta$ is transversal to all
strata. This
transversality condition is necessary for using Proposition \ref{transversal} and this is precisely the transversality condition that we need in Theorem \ref {theorem-1}.

 Let $p_i:M^{(r)}\rightarrow M$ be the $i^{th}$-projection,
then we have the holomorphic exterior product section
$$s=s_1\oplus\ldots\oplus s_r:M^{(r)}\rightarrow
p_1^*E_1\oplus\ldots \oplus p_r^*E_r,$$ given by
$s(x_1,\dots,x_r)=(s_1(x_1),\dots, s_r(x_r)).$ Then
$Z(s)=X_1\times\ldots \times X_r$ and $Z(\Delta^*(s))=X_{1}\cap
\ldots \cap X_{r}.$ Set $X=Z(\Delta^*(s)).$
The next result describes the total Schwartz-MacPherson class of $X$ in terms
of the total Schwartz-MacPherson classes of  the $X_i.$

\begin{proposition}\label{SM-inter}
$$c^{SM}(X)=  c\left( \left(
TM|_{X}\right)^{\oplus r-1} \right)^{-1} \, \cap \,\;\Big(c^{SM}(X_{1}) \cdot \ldots \cdot c^{SM}(X_{r})\Big).$$
 \end{proposition}

\begin{proof} By Theorem \ref{T:2} we have that
$$\Delta^{!}\left( \;c^{SM}(Z(s))\;\right)=c\left( \left( TM|_{Z(\Delta^{*}s)}\right)^{\oplus r-1}
\right)\cap c^{SM}(Z(\Delta^{*}s)).$$
Hence
$$c^{SM}(X))=c\left( \left( TM|_{X}\right)^{\oplus r-1}
\right)^{-1}\cap \Delta^{!}\left( c^{SM}(X_1\times\ldots \times X_r)\right).$$
Now, M. Kwieci\'{n}ski proved in \cite{Kw}  that Schwartz-MacPherson classes behave well with
respect to the exterior products, that is $$c^{SM}(X_1\times\ldots \times X_r)=c^{SM}(X_1)\times\ldots \times c^{SM}(X_r).$$
Hence $$\Delta^{!}\left(c^{SM}(X_1\times\ldots \times X_r)\right )=c^{SM}(X_1)\cdot\ldots \cdot c^{SM}(X_r)$$
\noindent and the result follows.

\end{proof}


\begin{remark}\label{Fulton 3.2.8}
In \cite[Example 3.2.8.]{Ful} it is proved the following: Let $Y$ and $Z$ be schemes, $p$ and $q$ the projections from $Y\times Z$ to $Y$ and $Z,$ $E$ and $F$ vector bundles on $Y$ and $Z,$
$\alpha \in H_*(Y,\Z)$ and $\beta\in H_*(Z,\Z).$ Then $$\left(c(E)\cap \alpha\right)\times \beta=c(p^*E)\cap(\alpha\times\beta)$$
\noindent and $$\left(c(E)\cap \alpha\right)\times\left(c(F)\cap\beta\right)=c(p^*E\oplus q^*F)\cap(\alpha\times\beta).$$

Since $c(p^*E)\cap((c(E)^{-1}\cap\alpha )\times\beta)=\alpha\times\beta$ we have that  $$(c(E)^{-1}\cap\alpha)\times\beta=c(p^*E)^{-1}\cap(\alpha\times\beta).$$
Analogously, $$\left(c(E)^{-1}\cap \alpha\right)\times\left(c(F)^{-1}\cap\beta\right)=c(p^*E\oplus q^*F)^{-1}\cap(\alpha\times\beta).$$
\end{remark}

 In \cite{O-Y} was stated without proof that Fulton-Johnson classes behave well with
respect to the exterior products.  For completeness we include it proof here

\begin{lemma}\label{FJ-product}
$$c^{FJ}(X_1\times \ldots \times X_r)=  c^{FJ}(X_{1}) \times \ldots \times c^{FJ}(X_{r}).$$
\end{lemma}

\begin{proof}
$$\begin{array}{lcl}
c^{FJ}(X) & = & c\left(TM^{(r)}|_{X_1\times \ldots \times X_r}\right)c\left(p_1^*E_1\oplus\ldots \oplus p_r^*E_r\right)^{-1}\cap \left[X_1\times \ldots \times X_r\right] \\
\\
 & = & c\left(p_1^*TM|_{X_1}\oplus \ldots \oplus p_r^*TM|_{X_r}\right)c\left(p_1^*E_1\oplus\ldots \oplus p_r^*E_r\right)^{-1}\cap \left(\left[X_1\right]\times \ldots \times \left[X_r\right] \right)\\
    \\
          & = & \left(c(TM|_{X_1})c(E_1)^{-1}\cap\left[X_1\right]\right)\times\ldots\times \left(c(TM|_{X_r})c(E_r)^{-1}\cap\left[X_r\right]\right) \\
    \\
                    & = & c^{FJ}(X_{1}) \times \ldots \times c^{FJ}(X_{r}).
\end{array}$$
\noindent where the third equality follows by Remark \ref{Fulton 3.2.8}.
\end{proof}

\begin{proposition}\label{FJ-inter}
$$c^{FJ}(X)=  c\left( \left(
TM|_{X}\right)^{\oplus r-1} \right)^{-1} \, \cap \,\;\Big(c^{FJ}(X_{1}) \cdot \ldots \cdot c^{FJ}(X_{r})\Big).$$
\end{proposition}

\begin{proof}
 By Proposition \ref{t:1} we have that
$$\Delta^{!}\left( \;c^{FJ}(Z(s))\;\right)=c\left( \left( TM|_{Z(\Delta^{*}s)}\right)^{\oplus r-1}
\right)\cap c^{FJ}(Z(\Delta^{*}s)).$$
Hence
$$c^{FJ}(X))=c\left( \left( TM|_{X}\right)^{\oplus r-1}
\right)^{-1}\cap \Delta^{!}\left( c^{FJ}(X_1\times\ldots \times X_r)\right).$$
By Lemma \ref{FJ-product} we have that
$$c^{FJ}(X))=c\left( \left( TM|_{X}\right)^{\oplus r-1}
\right)^{-1}\cap \Delta^{!}\left( c^{FJ}(X_1)\times\ldots \times c^{FJ}(X_r)\right).$$
Since $$\Delta^{!}\left( c^{FJ}(X_1)\times\ldots \times c^{FJ}(X_r)\right)=c^{FJ}(X_{1}) \cdot \ldots \cdot c^{FJ}(X_{r})$$
\noindent  the result follows.
\end{proof}

\begin{proaf}{\sc of Theorem \ref{theorem-1}:}
This follows immediately from Proposition \ref{SM-inter}  and Proposition \ref{FJ-inter}.
\end{proaf}

\begin{example} {\rm Let $Z_1$ and $Z_2$ be the hypersurfaces of $\mathbb{P}^4$ defined by $$H(x_0,\dots, x_4)=x_0 x_1 \ \quad  \hbox{and} \quad
G(x_0,\dots, x_4)=x_3 \;.$$
The line bundle of $Z_1$ is ${\mathcal O}(2H)$, where $H= c_1({\mathcal O}(1))$, so the class of the virtual tangent bundle of $Z_1$ is:
$$(1+H)^5 2H / (1+ 2H) = 2H + 6H^2 + 8H^3 + 4H^4,$$
while the Schwartz-MacPherson class is, by the inclusion-exclusion formula in \cite{Aluffi2}:
$$2 c(T\mathbb{P}^3) - c(T\mathbb{P}^2) = 2( (1+H)^4 H ) - (1+H)^3 H^2 = 2H + 7H^2 + 9H^3 + 5H^4$$
Therefore the Milnor class of $Z_1$ is $H^2 + H^3 + H^4$.
On the other hand, since $Z_2$ is smooth, the Schwartz-MacPherson class and the Fulton-Johnson class of $Z_2$ are $(1+H)^4 H= H+4H^2+6H^3+4H^4$.
Therefore, by Theorem \ref{theorem-1}, the Milnor class of $Z_1 \cap Z_2$ is given by
$${\mathcal M}(Z_1 \cap Z_2)=-c(T\mathbb{P}^4)^{-1}\cap c^{SM}(Z_2){\mathcal M}(Z_1)= -H^3.$$
}
\end{example}

\begin{remark} Take the
complete intersection $X=X_1 \cap X_2$, where $X_1$ is a smooth quadric surface in $\mathbb P^3$ and $X_2$ is a tangent plane get two distinct lines meeting at a point. The Milnor class of $X$ is simply the class of a point, but the Milnor classes of $X_1$ and $X_2$ are both zero because they both are smooth. This shows that a
transversality condition is necessary for our formula in \ref{main-lemma}.\end{remark}

\section{Applications to line bundles}

\begin{theorem} \label{main-lemma} With the conditions of Theorem \ref{theorem-1}
 we have:
 $${\mathcal M}(X)=(-1)^{nr-n}c\left( \left(
TM|_{X}\right)^{\oplus r-1} \right)^{-1}\cap
\sum\;(-1)^{(n-d_{1})\epsilon_{1}+\ldots+(n-d_{r})\epsilon_{r}}
 P_{1}\;\cdot\ldots\cdot\; P_{r}\in H_{*}(X),$$
where the sum runs over all choices  of $P_{i}\in \left\{{\mathcal
M}(X_i),c^{SM}(X_i)\right\},\,i=1,\dots,r,$  except
$(P_{1},\ldots,P_{r})=(c^{SM}(X_1),\ldots,c^{SM}(X_r))$ and where
$$\epsilon_{i}=\left\{\begin{array}{rcl} 1&,& \mbox{if} \;P_{i}=c^{SM}(X_i)\\
0&,& if \;P_{i}={\mathcal M}(X_i)\\
\end{array}\right. .$$

\end{theorem}

\begin{proof}
By Corollary \ref{10},
$$\Delta^{!}{\mathcal M}(Z(s))=(-1)^{nr-n}c\left( \left( TM|_{Z(\Delta^{*}s)}\right)^{\oplus
r-1} \right)\cap{\mathcal M}(Z(\Delta^{*}s)).$$ Thus,
$${\mathcal M}(X)=(-1)^{nr-n}c\left( \left(
TM|_{Z(\Delta^{*}s)}\right)^{\oplus r-1}
\right)^{-1}\cap\Delta^{!}{\mathcal M}(X_{1}\times \ldots \times
X_{r}),$$
 and using the description of the Milnor classes of a product due to   \cite[Corollary 3.1]{O-Y}, we have:
$${\mathcal M}(X)=(-1)^{nr-n}c\left( \left(
TM|_{X}\right)^{\oplus r-1} \right)^{-1}\cap
\;\sum(-1)^{(n-d_{1})\epsilon_{1}+\ldots+(n-d_{r})\epsilon_{r}}\Delta^{!}\left(
 P_{1}\times \ldots\times  P_{r}\right),$$
where the sum runs over all choices  of $P_{i}\in \left\{{\mathcal
M}(X_i),c^{SM}(X_i)\right\},\,i=1,\dots,r,$  except
$(P_{1},\ldots,P_{r})=(c^{SM}(X_1),\ldots,c^{SM}(X_r))$ and
where $$\epsilon_{i}=\left\{\begin{array}{rcl} 1&,& if \; P_{i}=c^{SM}(X_i)\\
0&,& if \; P_{i}={\mathcal M}(X_i)\\
\end{array}\right. .$$
The result follows because  $\Delta^{!}\left(
P_{1}\times \ldots\times P_{r}\right)= P_{1}\;\cdot\ldots\cdot\;
P_{r}\in H_{*}(X).$
\end{proof}

\begin{corollary}\label{11} $${\mathcal M}(X)= c\left( \left(
TM|_{X}\right)^{\oplus r-1} \right)^{-1}\cap
\displaystyle\sum_{i=1}^{r} (-1)^{D-d_i} a_{1,i}\cdot...\cdot
a_{r-1,i} \cdot {\mathcal M}(X_{i}),$$ where
$D=\displaystyle\sum_{j=1}^{r}d_j$ and
$a_{j,i}=\left\{\begin{array}{lcl}
c^{SM}(X_{j+1}) &\;{\rm if}\;& i \leq j\\
c^{FJ}(X_{j}) &\;{\rm if}\;& i> j\\
\end{array}\right. .$
\end{corollary}

 From now on we  replace the bundles $E_i$ by  line
bundles $L_i$.

\subsection{Aluffi type formula}  Let  $X=X_{1} \cap \ldots \cap X_{r}$ be as above.
The $\mu$-classes were
introduced  by P. Aluffi in \cite{Aluffi}.
For each $X_i$,  the Aluffi's
 $\mu$-class  of the
singular locus  is defined  by the formula
 $$\mu_{L_i}({\rm Sing}(X_i))=c(T^*M\otimes {L_i})\cap s({\rm Sing}(X_i),M),$$ where
$s({\rm Sing}(X_i),M)$ is the Segre class of ${\rm Sing}(X_i)$ in $M$ (see
 \cite[Chapter 4]{Ful}).
Given  a cycle $\alpha \in H_{2*}(X_i,\Z)$  and $\alpha=\sum_{j\geq
0}{\alpha}^j,$ where ${\alpha}^j$ is the  codimension $j$
component of $\alpha,$ then Aluffi introduced the following cycles
$$\alpha^{\vee}:=\sum_{j\geq 0}(-1)^j{\alpha}^j \quad  \hbox{and} \quad
 \alpha \otimes L_i:=\sum_{j\geq 0}\frac{{\alpha}^j}{c(L_i)^j} \;.$$
Then Aluffi proved in \cite{Aluffi} that the total Milnor class
${\mathcal M}(X_i)$ can be described as follows:

\begin{equation}\label{Aluffi-formula}{\mathcal M}(X_i)=(-1)^{n-1}c(L_i)^{n-1}\cap (\mu_{L_i}({\rm Sing}(X_i))^{\vee
}\otimes L_i).\end{equation}

 Again using Corollary \ref{11},  the above
equation yields:
\begin{corollary} \label{Aluffi} The Total Milnor class of  $X := X_1 \cap \ldots \cap X_r$ is:
$${\mathcal M}(X)=(-1)^{n-1}c\left( \left(
TM|_{X}\right)^{\oplus r-1} \right)^{-1}\cap    \left(\displaystyle\sum_{i=1}^{r} (-1)^{r-1} a_{1,i}\cdot\ldots\cdot a_{r-1,i}\cdot c(L_{i})^{n-1}\cap (\mu_{L_{i}}({\rm Sing}{(X_{i})\;})^{\vee
}\otimes L_{i})\right),$$
where  $a_{j,i}=\left\{\begin{array}{lcl}
c^{SM}(X_{j+1}) &\;{\rm if}\;& i \leq j\\
c^{FJ}(X_{j}) &\;{\rm if}\;& i> j\\
\end{array}\right. .$
\end{corollary}

\subsection{Parusi\'nski-Pragacz-type formula}
\label{section-P-P-formula}
We now assume each $X_i$ has a Whitney stratification  ${\mathcal S}_{i}$. One has in \cite{Par-Pr} the following characterization of the Milnor classes
of hypersurfaces in compact manifolds:

\begin{equation}\label{P}{\mathcal M}(X_i):=\sum_{S \in {\mathcal S}_{i}}\;\gamma_{S}\left(c(L_{i|_{X_i}})^{-1} \cap
c^{SM}(\overline{S})\right)\in H_*(X_i),\end{equation}
where
$\gamma_{S}$ is the function defined on each stratum $S$
as follows: for each $x \in S \subset X_i$, let $F_x$ be a
{\it local Milnor fibre}, and  let  $\chi(F_{x})$ be its Euler
characteristic. We set:
$$\mu(x;X_i):= (-1)^{n}\;(\chi(F_{x})-1) \,,$$ and
call it the {\it local Milnor number}. This number
is constant on each Whitney stratum,  so we denote it $\mu_{S}$. Then
$\gamma_{S}$ is defined inductively by:
$$\gamma_{S}=\mu_{S} - \sum_{S' \neq S,\;\overline{S'} \supset S}
\gamma_{S'}.$$

\begin{lemma}\label{3.2.8.-modificado}
Let $Y$ and $Z$ be subschemes of $M$, $W=Y\cap Z,$  $E$ a vector bundle on $Y,$
$\alpha \in H_*(Y,\Z)$ and $\beta\in H_*(Z,\Z).$ Then $$\left(c(E)\cap \alpha\right)\cdot \beta=c(E|_W)\cap(\alpha\cdot\beta).$$
\end{lemma}

\begin{proof}
Let $p$ be the projections from $Y\times Z$ to $Y$ and $d:W\rightarrow Y\times Z$ be the diagonal embedding.
 $$\begin{array}{lcl}
\left(c(E)\cap \alpha\right)\cdot \beta & = & d^!\left(\left(c(E)\cap \alpha\right)\times \beta \right) \\
\\
& = & d^!\left(c(p^*E)\cap \left(\alpha\times \beta \right)\right) \\
\\
& = & c(d^*p^*E)\cap d^!\left(\alpha\times \beta \right) \\
\\
& = & c(E|_W)\cap(\alpha\cdot\beta).
\end{array}$$
\noindent where the second and third equalities follows by Remark \ref{Fulton 3.2.8} and Remark \ref{Fu} (2) respectively.
\end{proof}

\begin{corollary}[Parusi\'{n}ski-Pragacz formula for local
complete intersections]\label{P-P} We have:
  $${\mathcal M}(X)=(-1)^{nr-n}c\left( \left(
 TM|_{X}\right)^{\oplus r-1} \right)^{-1} \cap \, \bigg(\sum{\alpha}_{S_{1},\dots,S_{r}}^{\epsilon_{1}, \dots, \epsilon_{r}}
 \frac{c(L_1)^{\epsilon_{1}}\cdot \ldots \cdot  c(L_r)^{\epsilon_{r}}}{c(L_1\oplus\ldots\oplus L_r)}\cap c^{SM}(\overline{S_1})\cdot \ldots \cdot  c^{SM}(\overline{S_r})\bigg)\,,$$

 \noindent where the sum runs over all possible choices  of the strata provided  $(S_{1},\dots,S_{r})\neq ((X_1)_{reg},\dots,(X_r)_{reg})$,
 $${\alpha}_{S_{1},\dots,S_{r}}^{\epsilon_{1}, \dots,
\epsilon_{r}}=(-1)^{(n-{1})(\epsilon_{1}+\ldots+\epsilon_{r})}
 \gamma_{S_1}^{1-\epsilon_{1}}\cdot \ldots \cdot
 \gamma_{S_r}^{1-\epsilon_{r}}\quad , \quad \hbox{and}  \quad
\epsilon_{i}=\left\{\begin{array}{rcl} 1&,& if \;S_{i}\subseteq (X_i)_{reg}\\
0&,& if \;\dim(S_i)<n-1\\
\end{array}\right. .$$
\end{corollary}

\begin{proof}
The proof will be by induction on $r.$ For $r=1$ this is Parusi\'{n}ski-Pragacz formula given in equation (\ref{P}) for $X_1.$ Let $Y=X_1\cap\ldots\cap X_{r-1}.$ Then $\dim Y=n-(r-1)$ and $X=Y\cap X_r.$
Hence $ {\mathcal M}(X)= {\mathcal M}(Y\cap X_r).$

By Theorem \ref{main-lemma} we have that
$$ {\mathcal M}(Y\cap X_r)=(-1)^{n}c\left( \left( TM|_{X}\right)
\right)^{-1}\cap \Big({\mathcal M}(Y)\cdot {\mathcal
M}(X_r)+ (-1)^{\dim Y} c^{SM}(Y)\cdot  {\mathcal
M}(X_r)+(-1)^{n-1}{\mathcal M}(Y)\cdot
c^{SM}(X_r)\Big).$$

By induction hypotheses, equation (\ref{P}) for $X_r,$ Proposition \ref{SM-inter} and Lemma \ref{3.2.8.-modificado} we have that

$${\tiny \begin{array}{l}
{\mathcal M}(X)  = (-1)^n c\left( \left( TM|_{X}\right)^{\oplus r-1}\right)^{-1}\cap \Big[(-1)^{nr} \bigg(\displaystyle\sum_{\underline{S}\neq \underline{X}}{\alpha}_{S_{1},\dots,S_{r-1}}^{\epsilon_{1}, \dots, \epsilon_{r-1}}
 \frac{c(L_1)^{\epsilon_{1}}\cdot \ldots \cdot  c(L_{r-1})^{\epsilon_{r-1}}}{c(L_1\oplus\ldots\oplus L_{r-1})}\cap \prod_{i=1}^{r-1}c^{SM}(\overline{S_i})\bigg)\cdot\bigg(\sum_{S_r}\;\gamma_{S_r}\left(c(L_{r|_{X_r}})^{-1} \cap
c^{SM}(\overline{S_r})\right)\bigg)
\\
\\\quad\quad\quad\quad\quad\quad\quad\quad\quad + (-1)^{n-r+1}\left(\prod_{i=1}^{r-1}c^{SM}(\overline{S_i})\right)\cdot \bigg(\sum_{S_r}\;\gamma_{S_r}\left(c(L_{r|_{X_r}})^{-1} \cap
c^{SM}(\overline{S_r})\right)\bigg)
\\
\\\quad\quad\quad\quad\quad\quad\quad\quad\quad + (-1)^{n-1+nr}\bigg(\displaystyle\sum_{\underline{S}\neq \underline{X}}{\alpha}_{S_{1},\dots,S_{r-1}}^{\epsilon_{1}, \dots, \epsilon_{r-1}}
 \frac{c(L_1)^{\epsilon_{1}}\cdot \ldots \cdot  c(L_{r-1})^{\epsilon_{r-1}}}{c(L_1\oplus\ldots\oplus L_{r-1})}\cap \prod_{i=1}^{r-1}c^{SM}(\overline{S_i})\bigg)\cdot c^{SM}(X_r)\Big]
\end{array}}$$
\noindent where $\underline{S}\neq \underline{X}$ means that $(S_{1},\dots,S_{r-1})\neq ((X_1)_{reg},\dots,(X_{r-1})_{reg}).$

Notice that $\gamma_{S_r}=0$ if $S_{r}\subseteq (X_r)_{reg}$ and that
$${\alpha}_{S_{1},\dots,S_{r}}^{\epsilon_{1}, \dots,
\epsilon_{r}}=\left\{\begin{array}{rcl} (-1)^{n-1}{\alpha}_{S_{1},\dots,S_{r-1}}^{\epsilon_{1}, \dots,
\epsilon_{r-1}}&,& if \;S_{r}\subseteq (X_r)_{reg}\\
\\
{\alpha}_{S_{1},\dots,S_{r-1}}^{\epsilon_{1}, \dots,
\epsilon_{r-1}}\cdot \gamma_{S_r}&,& if \;\dim(S_r)<n-1\\
\end{array}\right. .$$

Hence

$${\tiny \begin{array}{l}
{\mathcal M}(X)  = (-1)^{nr-n} c\left( \left( TM|_{X}\right)^{\oplus r-1}\right)^{-1}\cap \Big[\displaystyle\sum_{\underline{S}\neq \underline{X}\;\mbox{and}\; S_r\neq X_r}{\alpha}_{S_{1},\dots,S_{r}}^{\epsilon_{1}, \dots, \epsilon_{r}}
 \frac{c(L_1)^{\epsilon_{1}}\cdot \ldots \cdot  c(L_{r})^{\epsilon_{r}}}{c(L_1\oplus\ldots\oplus L_{r})}\cap \prod_{i=1}^{r}c^{SM}(\overline{S_i})
\\
\\\quad\quad\quad\quad\quad\quad\quad\quad\quad +\displaystyle\sum_{\underline{S}= \underline{X}\;\mbox{and}\; S_r\neq X_r}{\alpha}_{S_{1},\dots,S_{r}}^{\epsilon_{1}, \dots, \epsilon_{r}}
 \frac{c(L_1)^{\epsilon_{1}}\cdot \ldots \cdot  c(L_{r})^{\epsilon_{r}}}{c(L_1\oplus\ldots\oplus L_{r})}\cap \prod_{i=1}^{r}c^{SM}(\overline{S_i})
\\
\\\quad\quad\quad\quad\quad\quad\quad\quad\quad + \displaystyle\sum_{\underline{S}\neq \underline{X}\;\mbox{and}\; S_r= X_r}{\alpha}_{S_{1},\dots,S_{r}}^{\epsilon_{1}, \dots, \epsilon_{r}}
 \frac{c(L_1)^{\epsilon_{1}}\cdot \ldots \cdot  c(L_{r})^{\epsilon_{r}}}{c(L_1\oplus\ldots\oplus L_{r})}\cap \prod_{i=1}^{r}c^{SM}(\overline{S_i})\Big]
\end{array}}.$$

Now the result follows straightforwardly.

\end{proof}

\begin{remark}
[{\bf Milnor classes and global L\^e classes}]
In \cite {BMS} there is  a concept of global
L\^e classes of a singular hypersurface $Z$ in a smooth complex
submanifold $M$ of $\mathbb{P}^{N}$, and a formula  relating
these with the Milnor classes of $Z$.
The L\^e classes extend the notion of the local L\^e cycles introduced
 in \cite{Massey}. Using Corollary \ref{11} one gets
   also a
description of the Milnor classes of the local complete intersections
$X=X_{1} \cap \ldots \cap X_{r}$ via the L\^e classes of each
hypersurface $X_{i}$.
 \end{remark}

\vspace{0.5cm}
\

 \end{document}